\documentclass[12pt,a4paper]{amsart}
\usepackage{amsfonts}
\usepackage{amsthm}
\usepackage{amsmath}
\usepackage{amscd}
\usepackage[latin2]{inputenc}
\usepackage{t1enc}

\usepackage{indentfirst}
\usepackage{graphicx}
\usepackage{graphics}
\usepackage{pict2e}
\usepackage{epic}
\numberwithin{equation}{section}
\usepackage[margin=2.9cm]{geometry}
\usepackage{epstopdf} 
\usepackage{amscd}
\usepackage{sseq}
\usepackage{amsmath}
\usepackage{amssymb}
\usepackage{ dsfont }
\usepackage{textcomp}
\usepackage{amsthm}
\usepackage{tikz-cd}

\usepackage{graphicx}
\graphicspath{ {./images/} }

\usepackage{euscript}
\usepackage{ mathrsfs }

\newcommand{\set}[1]{\mathbb{#1}}

\usepackage{amsfonts}
\usepackage{amsrefs}
\usepackage{amsthm}
\usepackage{mathrsfs}
\usepackage{hyperref}
\usepackage{amsmath}

\theoremstyle{plain}
\newtheorem{Th}{Theorem}[section]

\newtheorem{Lemma}[Th]{Lemma}

 \theoremstyle{definition}

\newtheorem{?}[Th]{Problem}
\newtheorem{Ex}[Th]{Example}

\newcommand{\w}[0]{\omega}
\newcommand{\G}[0]{\Gamma}

\newcommand{\C}{\mathbb{C}}

\renewcommand{\P}{\mathbb{P}}
\newcommand{\g}[0]{\mathfrak{g}}
\newcommand{\h}[0]{\mathfrak{h}}
\newcommand{\sll}[0]{\mathfrak{sl}}
\newcommand{\slln}[0]{\mathfrak{sl}(n)}

\newcommand{\gln}[0]{\mathfrak{gl}(n)}
\newcommand{\om}[0]{\mathcal{O}_{min}}

\begin{document}

\title{Hikita conjecture for the minimal nilpotent orbit}

\author { Pavel Shlykov}

\address{Department of Mathematics, University of Toronto, Toronto, Ontario, Canada.}
\email{pavel.shlykov@mail.utoronto.ca}

\begin{abstract} We check that the statement of Hikita conjecture holds for the case of the minimal nilpotent orbit of a simple Lie algebra $\g$ of type ADE and $\C ^2 / \Gamma$.
\end{abstract}

\maketitle

\section{Introduction}

Symplectic duality is a hypothetical duality between conical symplectic singularities. At the moment, no rigorous definition exists, though there are a number of conjectured examples and a great number of connections that should tie together dual singularities. Namely, the Springer resolution is dual
to the Springer resolution for the Langlands dual group (\cite{BEGI}), hypertoric varieties are dual to other hypertoric varieties (\cite{TOR}),  finite ADE quiver varieties
are dual to slices in the affine Grassmannian for the Langlands dual group (\cite{ADE}) and many others. Not always it is known which symplectic resolution is dual to which, as the theory is a work in progress. The case we are interesed in is, however, known: the closure of the minilal orbit in a simply laced Lie algebra is dual to the Slodowy slice to the subregular orbit.

Now, if we have a pair of "dual" conical singularities, Hikita conjecture is a relation between the cohomology of one resolution and the coordinate ring of the other one. Namely, if $\widetilde{X}\rightarrow X$ and $\widetilde{X^*} \rightarrow X^*$ are a pair of dual conical symplectic resolutions and $T$ is a maximal torus of the Hamiltonian action on $\widetilde{X^*}$, there is an isomorphism between the cohomology ring of $\widetilde{X}$ and the coordinate ring of the fixed points scheme $(X^{*})^{T}$. Hikita observed it in his paper (\cite{HI}) for many of aforementioned cases and proved it for Hilbert scheme of points, finite type A quiver varieties and hypertoric varieties.

Let $\g$ be a simply laced simple Lie algebra. The closure of its minimal nilpotent orbit is expected to be dual to the Slodowy slice to the subregular orbit.

The Slodowy slice to the subregular orbit in a Lie algebra $\g$ is the same as $\C^2 / \Gamma$, where $\Gamma$ is a finite subgroup of $SL(2,\C)$ (corresponding to $\g$). It is a symplectic variety with rational double points. It admits a unique symplectic resolution $\widetilde{\C^2/\Gamma}$, given by the minimal resolution. The cohomology algebra of this resolution is known (we will prove it) to be $Sym^{\geq 2}[\h]$, where $\h$ is the abstract Cartan algebra of a Lie algebra $\g$, corresponding to $\Gamma$. 

The "dual" symplectic variety to it is given by the closure of the minimal nilpotent orbit in $\g$, or, equivalently (via the isomorphism), the  closure of the minimal orbit $\om$ in $\g^*$. We will work with the latter.

If we choose a generic action of $\C^*$, such that the fixed point scheme for it and for torus $T$ are same, Hikita conjecture for this pair of  singular symplectic varieties states that $$H^*(\widetilde{\C^2/\Gamma}) = \C [ \overline{\om}]^{\C^*}.$$

\section{}

We first prove the statement about the cohomology ring:
\begin{Lemma}
The cohomology ring of the resolution $H^*(\widetilde{\C^2/\Gamma})$ is equal to $Sym^{\geq 2}[\h]$, where $\h$ is the abstract Cartan algebra of $\g$. 
\end{Lemma} 
\begin{proof}

First, one notices that  on both $\C^2/\G$ and  $\widetilde{\C^2/\Gamma}$ there is an action of $\C^*$, that contracts the lower to the point zero. Thus, due to homotopy equivalence, we can restrict the computation of cohomology ring of $\widetilde{\C^2/\Gamma}$ to the fiber of the resolution over zero point of $\C^2/\G$. This is known to be a tree of $\P^1$'s that is a Dynkin diagram of the types ADE. 
Its cohomology ring can be seen to be given by $Sym^{\geq 2}[\h]$. 
First of all,  $\pi_1$ of it is clearly $0$, since we can retract every loop to a tree, formed by the points of intersection and connecting lines between them, and tree is contractible. We are left to deal with $H^2$. To do so, one can observe that our tree clearly has a decomposition into affine cells and dots, so we can obtain the generators for $H^2$ from each affine line. Since the number of $P^1$'s (and, thus, $\set{A}^1$'s) equals the $rk$ of corresponding Lie algebra $\g$ of type $ADE$, we thus obtain the cohomology ring, isomorphic to $Sym^{\geq 2}[\h]$, where $\h$ is the abstract Cartan algebra of $\g$. 

\end{proof}

Since the left part is known already what we are left to do is to check that the right hand side is the same. To do this it will be useful to simplify the problem as follows.

First of all one should notice that taking $\C^*$-invariants means the same as intersecting with the Cartan subalgebra -- $\overline{\om}^{\C^*}=\h \cap \overline{\om}$ as a scheme. So, instead of working with $\C^*$-invariant functions we can simply take the ideal of $\overline{\om}$ in $Sym[\g]$, look at the image of its projection in $Sym[\h]$ and factorize by it. The result of the factorization will be the ring we seek.

Now, let $\g$ be a simple Lie algebra, fix $\h$ a Cartan subalgebra and let $\overline{\om}$ denote the closure of the minimal nilpotent orbit in $\g^*$. The statement about the equality of algebras (with the reasoning above) will clearly follow from the following theorem.
\begin{Th}
Let I be the defining ideal of $\overline{\om}$ in $Sym[\g]$. Then its image under the projection 
$$Sym[\g] \twoheadrightarrow Sym[\h],$$
induced by the inclusion $\h^* \hookrightarrow \g^*$ is given by $Sym[\h]$ in degree $\geq 2$.
\end{Th}

\
Before moving to its proof we want to take a closer look at some simple special cases of this statement.
\begin{Ex}

Consider the easiest example possible: the case of $\sll (2)$. Since we have chosen $h$, we have both e and f and the nilpotent orbit (in this case there is only one nilpotent orbit apart from $0$) is given by the equation $h^2+ef=0$ (or, equivalently, by the ideal, generated by $h^2+ef$). This, after the projection to $Sym[\h]$ will give us $h^2$, which clearly generates the whole algebra in degree $\geq 2$.
\end{Ex}
\begin{Ex}

One more example would be the Lie algebra $\slln$ case. It is a bit more complicated -- at least there is be more than one nilpotent orbit. 

If a matrix belongs to the minimal nilpotent orbit its square is zero and its rank is 1. 
In terms of matrix equations such a matrix $A$ is given by the  $A^2=0, rk(A)=1$ -- every $2 \times 2$ minor should be zero and the matrix squared should be zero.
If we restrict those to the Cartan subalgebra we will get $a_{ii}^2=0$ (from the $A^2=0$) and $a_{ii}a_{jj}=0$ (from $det_{ij}=0$) which gives us all the functions in degree $2$ of $Sym[\h]$ for $ \gln $, thus in $\slln  $ too. 
\end{Ex}
\
\
\

To understand the theorem better we are going to use the following knowledge about adjoint representation of $\g$.
One should note that $\g^*$ is a representation of the type $V^*_{\theta}$ there $\theta$ is the highest weight for the adjoint representation. Moreover, $Sym^n[\g]$ can therefore be given as a sum $Sym^n[\g]= V(n\theta)\oplus L_n$ where $L_n$ stands for the sum of representations of lower weight. The fact is that the ideal we are interested in is actually given by the $\bigoplus_n L_n$. Indeed, every function from a representation of lower weight kills the highest weight vector $v_\theta$ from $\g^*=V^*(\theta)$. To find the generators of this ideal we can observe that it is actually given by  kernels of maps like $$V(n\theta) \otimes V(m\theta) \twoheadrightarrow V((n+m)\theta).$$
The objects of the form $V(n\theta)$ form a subring in the ring of all highest weight representations of our algebra $\g$.

The structure of generators of such a ring is given by the following theorem of Kostant (see \cite{E1}, \cite{LT}).

\begin{Th}
Let $\g$ be a semisimple Lie algebra and let $\Gamma_{\w_1}, \ldots, \Gamma_{\w_n}$ be irreducible representations corresponding to the fundamental weights. Let
$$A=Sym(\Gamma_{\w_1}, \ldots, \Gamma_{\w_n})$$.
This is a commutative graded algebra, and we can split it into pieces 
$$A^{a}=\bigoplus_{a_1,\ldots,a_n}Sym^{a_1}(\G_{\w_1})\otimes \ldots \otimes Sym^{a_n}(\G_{\w_n}),$$
where $a=(a_1, \ldots, a_n) $ is an $n$-tuple of nonnegative integers. $A^a$ then is the direct sum of the irreducible representation $\G_{\lambda}$, whose weight is given by $\lambda=\sum a_i\w_i$ and the sum $J^a$ of representations with strictly smaller highest weights. Thus, $J=\bigoplus J^a$ is an ideal in $A$, and it is generated by the elements of the form 
$$ \left ( \left [ \sum^n_{i=1}(x_i \otimes y_i+y_i \otimes x_i) \right ] - 2(\w_1, \w_2)Id \otimes Id) \right )   \cdot \  (v_1 \otimes v_2), $$
where $v_1 $and $v_2$ stand for vectors in 
$\G_{\w_1}$ and  $\G_{\w_2}$, and the sum  $\left [ \sum^n_{i=1}(x_i \otimes y_i+y_i \otimes x_i) \right ]$ denotes the Casimir element of the Lie algebra $\g$
\end{Th}

\ 
 
Namely, for our case the theorem  says, that the kernel of the morphism $\mathfrak{g^*}\otimes \mathfrak{g^*}=V(\theta) \otimes V(\theta) \rightarrow  V(2\theta)$, is generated by the elements of the form $C(v \cdot w) - 2 (\theta, \theta)v \cdot w$, where C is the Casimir element. Since we are interested in the image of the projection of this subspace on the $Sym^2[\h]$ we can take vector $h \in \h$ for both $v, w$ and see what happens to it. Namely if we choose a basis $e_i, f_i, h_j$ in $\g$, Casimir element can be written as a sum 
$$C=\sum_i (e_i \otimes f_i +f_i \otimes e_i) + \sum_j (h_i \otimes h_i).$$ 
The first part of this sum sends the element $h\otimes h$ to a kernel of the projection $Sym[\g] \rightarrow Sym[\h]$, whereas the second part acts by zero. So, after projection we will be left with $2(\theta, \theta)h\otimes h$ and elements like this clearly generate all the  $Sym^2[\h]$. The latter, however, clearly generates all the algebra $Sym[\h]$ in degree $\geq 2$. This proves the theorem $(2.2)$ which, in turn,  confirms the initial hypothesis.

\section*{Acknowledgements}
The author wanted to thank Alexander Braverman for many helpful discussions and remarks and for gradual replacement of sawdust with long words in author's head.


\begin{bibdiv}
\begin{biblist}

\bib{LT}{article}{
title={Representation-functors and flag-algebras for the classical groups },
author={Lancaster, Glenn },
author={Towler, Jacob},

journal= {Journal of algebra},
volume={59},
date={1979},
pages={16-38}
}


\bib{E1}{book}{
title={Representation theory},
subtitle={A first course},
author={Fulton, William},
author={Harris, Joe},
date={1991},
publisher={Springer-Verlag}

}

\bib{HI}{article}{
title={An algebra-geometric realization of the cohomology ring of Hilbert scheme of points in the affine plane},
author={Tatsuyuki, Hikita},
journal={Int. Math. Res. Not. no. 8, 2538-2561},
date={2017}}


\bib{BEGI}{article}{
title={Koszul duality
patterns in representation theory},
author={Beilinson, Alexander},
author={Ginzburg, Victor},
author={Soergel, Wolfgang},
journal={ J. Amer. Math. Soc. 9 no. 2, 473-527.},
date={1996}
}

\bib{TOR}{article}{
title={Hypertoric
category
O},
author={Braden, Tom},
author={Licata, Anthony},
author={Proudfoot, Nicholas},
author={Webster, Ben},
journal={Adv. Math. 231 (2012), no. 3-4, 1487-1545}
}


\bib{ADE}{article}{
title={Quantizations of conical
symplectic resolutions I: local and global structure},
author={Braden, Tom},
author={Proudfoot, Nicholas},
author={Webster, Ben},
journal={Ast erisque (2016), no. 384,
1-73.}
}





















\end{biblist}
\end{bibdiv}
\end{document}